\documentclass[preprint,12pt,sort&compress,numbers]{elsarticle}

\usepackage{amsmath,amsthm,amsfonts,amssymb,latexsym,mathrsfs,url,enumerate}

\newtheorem{thm}{Theorem}[section]
\newtheorem{lem}[thm]{Lemma}
\newtheorem{coro}[thm]{Corollary}
\newtheorem{prop}[thm]{Proposition}

\newtheorem*{CIT}{Cauchy Interlacing Theorem}

\newtheorem*{IP}{Inclusion Principle}

\theoremstyle{definition}

\newtheorem{rem}[thm]{Remark}

\def\diag{{\rm diag}}

\def\H{\mathcal{H}}
\def\la{\lambda}
\def\intrz{\lessdot}
\def\comp{\bowtie}

\def\ol{\overline}

\def\L{\mathcal{L}}
\def\n{\mathfrak{n}}

\newcommand{\w}[1]{\widetilde{#1}}

\journal{Linear and Multilinear Algebra}
\usepackage{geometry}
\begin{document}

\begin{frontmatter}

\title{Inertia indices and eigenvalue inequalities for Hermitian matrices}

\author[a]{Sai-Nan Zheng}
\ead{zhengsainandlut@hotmail.com}
\author[b]{Xi Chen}
\ead{chenxi@dlut.edu.cn}
\author[c]{Lily Li Liu}
\ead{liulily@qfnu.edu.cn}
\author[b]{Yi Wang}
\ead{wangyi@dlut.edu.cn}

\address[a]{School of Mathematics, Dongbei University of Finance and Economics, Dalian 116025, P.R. China}
\address[b]{School of Mathematical Sciences, Dalian University of Technology, Dalian 116024, P.R. China}
\address[c]{School of Mathematical Sciences, Qufu Normal University, Qufu 273165, P.R. China}
\begin{abstract}
We present a characterization of eigenvalue inequalities between two Hermitian matrices
by means of inertia indices.
As applications,
we deal with some classical eigenvalue inequalities for Hermitian matrices,
including the Cauchy interlacing theorem and the Weyl inequality,
in a simple and unified approach.
We also give a common generalization of eigenvalue inequalities for
(Hermitian) normalized Laplacian matrices of simple (signed, weighted, directed) graphs.
Our approach is also suitable for Hermitian matrices of the second kind of digraphs
recently introduced by Mohar.
\end{abstract}


\end{frontmatter}

\section{Introduction}
A univariate polynomial is {\it real-rooted} if all of its coefficients and roots are real.
Let $f$ be a real-rooted polynomial of degree $n$.
Denote its roots by
$r_1(f)\ge r_2(f)\ge\cdots\ge r_n(f)$.
For convenience we set that $r_i(f)=+\infty$ for $i<1$ and $r_i(f)=-\infty$ for $i>n$.
Let $f$ and $g$ be two real-rooted polynomials of degree $n$ and $m$ respectively.
We say that $f$ {\it interlaces} $g$, denoted by $f\intrz g$,
if $n\le m\le n+1$ and $r_{i}(g)\ge r_{i}(f)\ge r_{i+1}(g)$ for all $i$.
We say that $f(x)$ and $g(x)$ are {\it compatible},
denoted by $f\comp g$,
if $|m-n|\le 1$ and $r_{i-1}(g)\ge r_{i}(f)\ge r_{i+1}(g)$ for all $i$.
Chudnovsky and Seymour \cite{CS07} introduced the concept of compatible polynomials and
showed that the independence polynomial of a claw-free graph has only real roots.
It is easy to see that $f\comp g$ is equivalent to $g\comp f$
and that $f\intrz g$ implies $f\comp g$.
We refer the reader to
\cite{CS07,Fis05,Liu12,LW07,MSS15,MSS15-2,MSS18,WY05,WZ19}
for further information about interlacing and compatible polynomials.

Let $A$ and $B$ be two Hermitian matrices.
We will use the notation $A\intrz B$ and $A\comp B$ if their characteristic polynomials
$\det(\la I-A)\intrz \det(\la I-B)$ and $\det(\la I-A)\comp \det(\la I-B)$ respectively.
In matrix analysis and spectral graph theory,
one frequently encounters Hermitian matrices whose characteristic polynomials are interlacing or compatible.
Throughout this paper we use $\H_n$ to denote the set of $n\times n$ Hermitian matrices.
For a Hermitian matrix $H\in\H_n$,
we arrange its eigenvalues in a nonincreasing order:
$\la_1(H)\ge\la_2(H)\ge\cdots\ge\la_n(H)$.
Let $n_+(H)$ ($n_-(H)$, resp.) denote
the positive (negative, resp.) inertia index,
i.e.,
the number of positive (negative, resp.) eigenvalues of $H$.
In this paper
we present the following characterization of eigenvalue inequalities between two Hermitian matrices
by means of inertia indices
and then apply it to investigate when the characteristic polynomials of Hermitian matrices interlace or are compatible.

\begin{thm}\label{cri-thm}
Let $A,B$ be two Hermitian matrices and $m\in\mathbb{Z}$. Then
$\la_{i+m}(B)\le\la_{i}(A)$ for all $i$ if and only if
$n_+(B-rI)-n_+(A-rI)\le m$ for every $r\in\mathbb{R}$.
\end{thm}
\begin{proof}
Assume that $\la_{i+m}(B)\le\la_{i}(A)$ for all $i$.
Then for arbitrary $r\in\mathbb{R}$,
\begin{equation}\label{i,i+m}
\la_{i+m}(B-rI)=\la_{i+m}(B)-r\le\la_i(A)-r=\la_i(A-rI).
\end{equation}
Let $n_+(A-rI)=p$.
Then $\la_{p+1}(A-rI)\le0$, and so $\la_{p+m+1}(B-rI)\le0$ by \eqref{i,i+m}.
It follows that $n_+(B-rI)\le p+m$, i.e.,
$n_+(B-rI)-n_+(A-rI)\le m$.

Conversely, assume that there exists an index $i_0$ such that $\la_{{i_0}+m}(B)>\la_{i_0}(A)$.
Take $r_0\in(\la_{i_0}(A),\la_{{i_0}+m}(B))$.
Then $\la_{i_0}(A-r_0I)<0$ and $\la_{{i_0}+m}(B-r_0I)>0$.
Thus $n_+(A-r_0I)\le i_0-1$ and $n_+(B-r_0I)\ge i_0+m$,
which implies that $n_+(B-r_0I)-n_+(A-r_0I)\ge m+1$.
In other words,
if $n_+(B-rI)-n_+(A-rI)\le m$ for every $r\in\mathbb{R}$,
then $\la_{i+m}(B)\le\la_{i}(A)$ for all $i$.
This completes the proof of the theorem.
\end{proof}

\begin{rem}\label{fin}
Clearly, both $n_+(A-rI)$ and $n_+(B-rI)$ are right continuous staircase functions of $r\in\mathbb{R}$,
so is their difference $n_+(A-rI)-n_+(B-rI)$.
Thus, for arbitrary real number $r_0$,
we have $n_+(A-rI)-n_+(B-rI)\equiv n_+(A-r_0I)-n_+(B-r_0I)$
in a certain right neighbourhood of $r_0$.
As a consequence,
it is impossible that an inequality for $n_+(A-rI)-n_+(B-rI)$ is false
only for finite many real numbers $r$.
\end{rem}

\begin{rem}\label{rem-mat}
The following are two particularly interesting special cases of Theorem \ref{cri-thm}.
\begin{enumerate}
  \item[\rm (i)] $A\intrz B$ if and only if $0\le n_+(B-rI)-n_+(A-rI)\le 1$ for every $r\in\mathbb{R}$.
  \item[\rm (ii)] $A\comp B$ if and only if $|n_+(B-rI)-n_+(A-rI)|\le 1$ for every $r\in\mathbb{R}$.
\end{enumerate}
\end{rem}

In \S 2 we use Theorem \ref{cri-thm} to
deal with some classical eigenvalue inequalities for Hermitian matrices,
including the Cauchy interlacing theorem and the Weyl inequality.
Various Hermitian matrices are often involved in spectral graph theory,
such as (Hermitian) adjacency matrices and normalized Laplacian matrices of simple graphs and digraphs.
In \S 3 we use Theorem \ref{cri-thm} to give a common generalization of eigenvalue inequalities for
(Hermitian) normalized Laplacian matrices of simple (signed, weighted, directed) graphs.

\section{Matrix ananlysis}

We first use Theorem \ref{cri-thm} to give a simple proof of the following result,
which is known as the inclusion principle \cite[Theorem 4.3.28]{HJ13}.

\begin{IP}
Let
$A=
\left(
  \begin{array}{cc}
    B_{k\times k} & C \\
    C^* & D \\
  \end{array}
\right)\in\H_n$.
Then $\la_{n-k+i}(A)\le\la_i(B)\le\la_i(A)$.
\end{IP}
\begin{proof}
Let $r$ be a real number such that $\det (B-rI)\neq 0$.
Then
$$A-rI=
\left(
  \begin{array}{cc}
    B-rI & C \\
    C^* & D-rI \\
  \end{array}
\right)
\cong
\left(
  \begin{array}{cc}
    B-rI & 0 \\
    0 & D-rI-C^*(B-rI)^{-1}C \\
  \end{array}
\right),$$
where $\cong$ denotes the congruence of matrices,
and so
\begin{equation}\label{n-AB}
  n_+(B-rI)\le n_+(A-rI)\le n_+(B-rI)+n-k.
\end{equation}
Clearly, there are only finite many real numbers $r$ such that $\det (B-rI)=0$.
Hence \eqref{n-AB} holds for all $r\in\mathbb{R}$.
Thus $\la_{n-k+i}(A)\le\la_i(B)\le\la_i(A)$ by Theorem \ref{cri-thm}.
\end{proof}

The Cauchy interlacing theorem is a special case of the inclusion principle.

\begin{CIT}
Let
$A=
\left(
  \begin{array}{ll}
    B & \alpha \\
    \alpha^* & a \\
  \end{array}
\right)\in\H_n$,
where $\alpha\in\mathbb{C}^{n-1}$ and $a\in\mathbb{R}$.
Then
$B\intrz A$.
\end{CIT}

The following result is an immediate consequence of Theorem \ref{cri-thm} and the inclusion principle,
which can also be proved by a direct argument
(we leave the proof to the interested reader).

\begin{prop}\label{pii+}
Let $A,B\in\H_n$.
Then $n_+(A+B)\le n_+(A)+n_+(B)$.
\end{prop}
\begin{proof}
The statement follows from
$\left(
   \begin{array}{cc}
     A & O \\
     O & B \\
   \end{array}
 \right)
\cong
\left(
   \begin{array}{cc}
     A+B & B \\
     B & B \\
   \end{array}
 \right)$
and \eqref{n-AB}.
\end{proof}

\begin{prop}\label{pii-}
Let $A,B\in\H_n$ and $m\in\mathbb{Z}$.
\begin{enumerate}
  \item[\rm (i)] If $n_+(B)\le m$, then $\la_{i+m}(A+B)\le\la_{i}(A)$ for all $i$. 
  \item[\rm (ii)] If $n_-(B)\le m$, then $\la_{i+m}(A)\le\la_{i}(A+B)$ for all $i$.
\end{enumerate}
\end{prop}
\begin{proof}
For every $r\in\mathbb{R}$,
we have
$n_+(A+B-rI)-n_+(A-rI)\le n_+(B)$
by Proposition \ref{pii+}.
Thus (i) follows from Theorem \ref{cri-thm}.
And (ii) follows from (i) by noting that $n_+(-B)=n_-(B)$.
\end{proof}

\begin{coro}[Monotonicity Theorem]\label{mt}
Suppose that $A, B\in\H_n$ and $B$ is positive semi-definite.
Then $\la_i(A)\le\la_i(A+B)$.
\end{coro}

\begin{coro}
\label{weyl}
Let $A, B\in\H_n$.
If $n_+(B)\le p$ and $n_-(B)\le q$,
then
$$\la_{i+q}(A)\le\la_i(A+B)\le\la_{i-p}(A).$$
\end{coro}

\begin{coro}[Weyl Inequality]
Let $A, B\in\H_n$.
Then
$\la_{i+j-1}(A+B)\le\la_i(A)+\la_j(B)$.
\end{coro}
\begin{proof}
Clearly, $n_+(B-\la_j(B)I)\le j-1$.
Thus $\la_{i+j-1}(A+B-\la_j(B)I)\le\la_i(A)$ by Proposition~\ref{pii-},
i.e., $\la_{i+j-1}(A+B)\le\la_i(A)+\la_j(B)$.
\end{proof}

Corollary \ref{weyl} is equivalent to the Weyl inequality
(see \cite{WZ19} for instance),
and is often more convenient to use.
A particular interesting special case of Corollary \ref{weyl} is the interlacing theorem
\cite[Corollary 4.3.9]{HJ13}
for a rank-one Hermitian perturbation of a Hermitian matrix.

\begin{coro}[Interlacing Theorem]
Let $A\in\H_n$ and $\alpha\in\mathbb{C}^n$.
Then $A\intrz (A+\alpha\alpha^*)$, i.e.,
$$\la_{i}(A)\le\la_i(A+\alpha\alpha^*)\le\la_{i-1}(A).$$
\end{coro}

Another useful special case of Corollary \ref{weyl}
is the following result about compatible polynomials.

\begin{coro}
\label{comp}
Let $A, B\in\H_n$.
If $n_+(B)=n_-(B)=1$,
then $A\comp (A+B)$, i.e.,
$$\la_{i+1}(A)\le\la_i(A+B)\le\la_{i-1}(A).$$
\end{coro}

The Cauchy interlacing theorem states that
a Hermitian matrix interlaces its bordered Hermitian matrix.
More generally,
suppose that $P,A\in\H_n$ and define
$$f(\la;P,A)=\det(\la P-A).$$
Then $f(\la;I,A)$ is precisely the characteristic polynomial of $A$.
For a real-rooted polynomial $f$,
let $n_+(f), n_-(f)$ and $n_0(f)$ denote the number of positive, negative and zero roots of $f$ respectively.
Call $(n_+(f),n_-(f), n_0(f))$ the {\it inertia index} of $f$.
Clearly, if $f$ is the characteristic polynomial of a Hermitian matrix $A$,
then the inertia index of $f$ coincides with that of $A$.
We have the following result.

\begin{coro}\label{A,B}
Suppose that $P,A\in\H_n$ and $P$ is positive definite. Then
\begin{enumerate}
  \item[\rm (i)] $f(\la;P,A)$ is a real-rooted polynomial in $\la$;
  \item[\rm (ii)] $f(\la;P,A)$ has the same inertia index as $A$;
  and
  \item[\rm (iii)] $f(\la;\ol{P},\ol{A})\intrz f(\la;P,A)$,
  where $\ol{P}$ (resp., $\ol{A}$) is the matrix obtained from $P$ (resp., $A$) by deleting the last row and column.
\end{enumerate}
\end{coro}
\begin{proof}
Let $B=P^{-1/2}AP^{-1/2}$.
Then $B$ is a Hermitian matrix and is congruent with $A$.
Moreover, $f(\la;P,A)=\det(\la P-A)
=\det(P)\det(\la I-B)$.
So $f(\la;P,A)$ has the same roots as the characteristic polynomial $\det(\la I-B)$ of $B$.
Thus (i) and (ii) follow.

Similarly, let $C=\ol{P}^{-1/2}\ol{A}~\ol{P}^{-1/2}$.
Then $f(\la;\ol{P},\ol{A})$ has the same roots as $\det(\la I-C)$.
So, to prove (iii), it suffices to prove that $C\intrz B$.
We prove it by Remark~\ref{rem-mat}~(i).
Let $r\in\mathbb{R}$.
Note that $B-rI=P^{-1/2}(A-rP)P^{-1/2}$.
Hence $n_+(B-rI)=n_+(A-rP)$.
Similarly, $n_+(C-rI)=n_+(\ol{A}-r\ol{P})$.
Since $\ol{A}-r\ol{P}$ is the $(n-1)\times (n-1)$ principal submatrix of $A-rP$,
we have $0\le n_+(A-rP)-n_+(\ol{A}-r\ol{P})\le 1$ by the Cauchy interlacing theorem and Remark~\ref{rem-mat}~(i).
Thus $0\le n_+(B-rI)-n_+(C-rI)\le 1$,
and so $C\intrz B$ again by Remark~\ref{rem-mat}~(i).
\end{proof}

\begin{rem}
Let $\ol{P}_i$ (resp., $\ol{A}_i$) denote the matrix obtained from $P$ (resp., $A$) by deleting the $i$-th row and column
and $f_i(\la)=f(\la;\ol{P}_i,\ol{A}_i)$.
Then each $f_i(\la)$ interlaces $f(\la;P,A)$.
It follows that $\sum_ic_if_i(\la)$ is real-rooted for all $c_i\ge 0$
(see \cite[Theorem 3.6]{CS07} for details).
\end{rem}

\begin{rem}
Corollary \ref{A,B} (i) and (iii) also hold when $P$ is positive semi-definite by a standard continuity argument.
We refer the reader to \cite{BB08} for some related results.
\end{rem}

\section{Spectral graph theory}

Let $G=(V,E)$ be a simple graph with $n$ vertices $v_1,\ldots,v_n$ and edge set $E\subseteq V\times V$.
The {\it adjacency matrix} $A(G)=(a_{ij})_{n\times n}$ of $G$ is defined by
$$a_{ij}=\left\{
\begin{array}{ll}
1, & \hbox{if $(v_i,v_j)\in E$;} \\
0, & \hbox{otherwise.}
\end{array}
\right.$$
Let $D=\diag(d_1,\ldots,d_n)$ be the {\it degree matrix} of $G$,
where $d_i=\deg v_i$.
The matrix $L(G)=D-A(G)$
is called the {\it Laplacian} matrix.
We refer the reader to
\cite{AM85,Mer95}
for further information about the Laplacian matrix.
The {\it normalized Laplacian} matrix of $G$ is defined as
$\L(G)=D^{-1/2}L(G)D^{-1/2}$
with the convention that the $i$th diagonal entry of $D^{-1}$ is $0$ if $d_i=0$.
For any vertex $v$ of $G$,
it is clear that $A(G-v)\intrz A(G)$ by the Cauchy interlacing theorem.
For any edge $e$ of $G$,
Chen et al. \cite{Chen04}
showed that
$L(G-e)\intrz L(G)$ and $\L(G-e)\comp\L(G)$
(see \cite{Li} for a short proof).

There are some similar results for signed graphs and weighted graphs.
A {\it signed graph} $G^\sigma$ consists of a simple graph $G$
and a map $\sigma: E\longrightarrow\{+1,-1\}$.
The {\it signed adjacency matrix} $A^\sigma(G)=(a_{ij}^\sigma)_{n\times n}$
is defined by $a_{ij}^\sigma=\sigma(v_i,v_j)a_{ij}$
and the degree matrix is still $D=\diag(d_1,\ldots,d_n)$.
Define the {\it signed Laplacian} $L^\sigma(G)=D-A^\sigma(G)$
and the {\it normalizer} $\L^\sigma(G)=D^{-1/2}L^\sigma(G)D^{-1/2}$.
Clearly, $L^\sigma=L$ if $\sigma\equiv 1$.
Atay and Tun\c{c}el \cite[Theorem 8]{AT14} showed that $\L^\sigma(G-e)\comp \L^\sigma(G)$ for any $e\in E$.
A {\it weighted graph} $(G,w)$ is a graph $G$ (possibly with loops)
with a nonnegative weight function $w: V\times V\longrightarrow [0,\infty)$
with $w(u,v)=w(v,u)$ and $w(u,v)>0$ if and only if there is an edge joining $u$ and $v$.
The {\it adjacency matrix} is defined by $a_{ij}=w(v_i,v_j)$.
The {\it diagonal degree} matrix is defined by $d_i=\sum_{v_j\thicksim v_i}w(v_i,v_j)$.
The Laplacian $L(G,w)$ and its normalizer $\L(G,w)$ is similarly defined as above.
We say that $(H,w_H)$ is a {\it subgraph} of $(G,w_G)$
if $H$ is a subgraph of $G$ and $w_H(e)\le w_G(e)$ for all $e\in E(H)$.
In this case, we define the weighted graph $G-H$ with the weight function $w_{G-H}=w_G-w_H$.
Let $e\in E(G)$ and $H=\{e\}$.
Butler \cite{Butler} showed that $\L(G-H,w_{G-H})\comp \L(G,w_G)$.

Recently,
Yu et al. \cite{YDEJ} considered the case of
simple directed graphs.
A {\it directed graph} $X$ consists of a finite set $V=\{v_1,\ldots,v_n\}$ of vertices
together with a subset $E\subseteq V\times V$ of ordered pairs
called {\it arcs} or {\it directed edges}.
If $(u,v)\in E$ and $(v,u)\in E$,
we say that the unordered pair $[u,v]$ is a {\it digon} of $X$.
Following \cite{LL15, GM17},
define the Hermitian adjacency matrix $H(X)=(h_{ij})_{n\times n}$ of $X$ by
\begin{equation*}\label{h}
h_{ij}=
\left\{
  \begin{array}{rl}
    1, & \hbox{if $(v_i,v_j)\in E$ and $(v_j,v_i)\in E$;} \\
    i, & \hbox{if $(v_i,v_j)\in E$ and $(v_j,v_i)\not\in E$;} \\
    -i, & \hbox{if $(v_i,v_j)\not\in E$ and $(v_j,v_i)\in E$;} \\
    0, & \hbox{otherwise.}
  \end{array}
\right.
\end{equation*}
Following \cite{YDEJ},
define the {\it Hermitian Laplacian} matrix 
$L(X)=D-H(X)$
and the {\it Hermitian normalized Laplacian} matrix 
$\L(X)=D^{-1/2}L(X)D^{-1/2}$,
where $D=\diag(d_1,\ldots,d_n)$ is the degree matrix of the corresponding undirected graph.
Yu et al. \cite[Theorem 3.6]{YDEJ} showed that $\L(X-e)\comp\L(X)$
for any arc or digon $e$ of $X$.

More recently,
Mohar \cite{Moh20} introduced a new kind of Hermitian matrix for digraphs.
Denote by $\omega=(1+i\sqrt{3})/2$ the primitive sixth root of unity
and let $\overline{\omega}$ be its conjugate.
Following Mohar \cite{Moh20},
define the Hermitian adjacency matrix $\w{H}(X)=[\w{h}_{ij}]_{n\times n}$
of the second kind of $X$ by
$$
\w{h}_{ij}=
\left\{
  \begin{array}{rl}
    1, & \hbox{if $(v_i,v_j)\in E$ and $(v_j,v_i)\in E$;} \\
    \omega, & \hbox{if $(v_i,v_j)\in E$ and $(v_j,v_i)\not\in E$;} \\
    \overline{\omega}, & \hbox{if $(v_i,v_j)\not\in E$ and $(v_j,v_i)\in E$;} \\
    0, & \hbox{otherwise.}
  \end{array}
\right.
$$
Define
the corresponding Hermitian Laplacian matrix
and Hermitian normalized Laplacian matrix by
$\w{L}(X)=D-\w{H}(X)$
and
$\w{\L}(X)=D^{-1/2}\w{L}(X)D^{-1/2}$ respectively.

Note that for the simple graph $G$ and $e\in E$, we have
$$L(G)-L(G-e)=
  \left(
  \begin{array}{rr}
    1 & -1 \\
    -1 & 1 \\
  \end{array}
  \right)\oplus 0.$$
We use this notation as short hand for
$$L(G)-L(G-e)=E_{u,u}-E_{u,v}-E_{v,u}+E_{v,v},$$
where $e=(u,v)$ and $E_{u,v}$ is the matrix with $1$ at the $(u,v)$ position and $0$ elsewhere.

For the signed graph $G^\sigma$ and the weighted graph $(G,w)$, we have
$$L^\sigma(G)-L^\sigma(G-e)=
\left(
\begin{array}{cc}
1 & -\sigma(e) \\
-\sigma(e) & 1 \\
\end{array}
\right)\oplus 0$$
and
$$L(G,w_G)-L(G-H,w_{G-H})=
  w_H(e)\left(
  \begin{array}{rr}
    1 & -1 \\
    -1 & 1 \\
  \end{array}
  \right)\oplus 0,$$
where $H=\{e\}$.
Similarly, if $X$ is a directed graph and $e$ is a digon or a directed edge of $X$,
then
$$L(X)-L(X-e)=
  \left(
  \begin{array}{rr}
    1 & -1 \\
    -1 & 1 \\
  \end{array}
  \right)\oplus 0
\textrm{ or }
  \left(
  \begin{array}{rr}
    1 & \pm i \\
    \mp i & 1 \\
  \end{array}
  \right)\oplus 0,$$
and
$$\w{L}(X)-\w{L}(X-e)=
  \left(
  \begin{array}{rr}
    1 & -1 \\
    -1 & 1 \\
  \end{array}
  \right)\oplus 0
\textrm{ or }
  \left(
  \begin{array}{rr}
    1 & \pm \omega \\
    \pm \overline{\omega} & 1 \\
  \end{array}
  \right)\oplus 0.$$

A Hermitian matrix $L=(\ell_{ij})_{n\times n}$ is called a {\it generalized Laplacian} matrix if
$\ell_{ii}\ge \sum_{j\neq i}|\ell_{ij}|$ for $i=1,\ldots,n$.
For such a matrix,
define its {\it normalizer} $\L=D^{-1/2}LD^{-1/2}$,
where $D$ is the diagonal matrix $\diag(d_{1},\ldots,d_{n})$ such that
$d_i=\ell_{ii}$ if $\ell_{ii}\neq 0$ and $d_i=0$ otherwise.
Clearly, $L(G),L^{\sigma}(G),L(G,w_G), L(X), \w{L}(X)$ are all generalized Laplacian matrices.
The following result is a common generalization of eigenvalue inequalities for
(normalized) Laplacians of simple (signed, weighted, directed) graphs.

\begin{prop}\label{dif}
Let $L_1$ and $L_2$ be two $n\times n$ generalized Laplacian matrices such that
\begin{equation}\label{dif-L}
L_1-L_2=
w\left(
  \begin{array}{rr}
    1 & c \\
    \overline{c} & 1 \\
  \end{array}
\right)\oplus 0,
\end{equation}
where $w>0, c\in\mathbb{C}$ and $|c|=1$.
Then $L_2 \intrz L_1$ and $\L_2\comp\L_1$.
\end{prop}

To prove Proposition \ref{dif}, we first establish the following result.

\begin{lem}\label{n=n}
Let $L$ be a generalized Laplacian matrix and
$\L=D^{-1/2}LD^{-1/2}$. 
Then
\begin{enumerate}
  \item[\rm (i)] the eigenvalues of $\L$ are in the interval $[0,2]$, and
  \item[\rm (ii)] $n_+(\L-rI)=n_+(L-rD)$ for $r\ge 0$.
\end{enumerate}
\end{lem}
\begin{proof}
Clearly,
if the $i$th diagonal entry of $L$ is $0$,
then all entries in the $i$th row (column) of $L$ are $0$.
Assume now that $L$ has precisely $k$ zero diagonal entries.
Delete the corresponding $k$ (zero) rows and $k$ (zero) columns of $L$ ($D$, $\L$, resp.)
to obtain $\ol{L}$ ($\ol{D}$, $\ol{\L}$, resp.).
Then $\ol{L}$ is a generalized Laplacian matrix with nonzero diagonal entry
and $\ol{\L}=\ol{D}^{~{-1/2}}\ol{L}~\ol{D}^{~{-1/2}}$.

(i)\quad
It suffices to consider the case that all diagonal entries of $L$ are nonzero.
In this case,
$\L$ and $D^{-1}L$ have the same eigenvalues since they are similar: $D^{-1}L=D^{-1/2} \L D^{1/2}$.
Note that the diagonal entries of $D^{-1}L$ are all $1$.
Hence the eigenvalues of $D^{-1}L$ satisfy $|\la-1|\le 1$
by the Gershgorin circle theorem
(see \cite[Theorem 6.1.1]{HJ13} for instance),
i.e., $\la \in [0,2]$.
Thus eigenvalues of $\L$ are all in $[0,2]$.

(ii)\quad
Note that the characteristic polynomials
$$\det (\la I-(L-rD))=\la^k\det (\la I-(\ol{L}-r\ol{D}))$$
and
$$\det (\la I-(\L-rI))=(\la+r)^k\det (\la I-(\ol{\L}-rI)).$$
Hence $n_+(L-rD)=n_+(\ol{L}-r\ol{D})$ and
$n_+(\L-rI)=n_+(\ol{\L}-rI)$ for $r\ge 0$.
On the other hand,
$$\ol{\L}-rI=\ol{D}^{\ -1/2}(\ol{L}-r\ol{D}))\ol{D}^{\ -1/2},$$
and so $n_+(\ol{\L}-rI)=n_+(\ol{L}-r\ol{D})$.
Thus we conclude that $n_+(\L-rI)=n_+(L-rD)$ for $r\ge 0$.
\end{proof}

We are now in a position to prove Proposition \ref{dif}.

\begin{proof}[Proof of Proposition \ref{dif}]
Denote $\Delta=\left(
  \begin{array}{rr}
    1 & c \\
    c^* & 1 \\
  \end{array}
\right)$.
Clearly, two eigenvalues of $\Delta$ are $0$ and $2$.
Thus $n_+(L_1-L_2)=n_+(\Delta)=1$ and $n_-(L_1-L_2)=n_-(\Delta)=0$ by \eqref{dif-L},
and therefore $L_2\intrz L_1$ by Corollary \ref{weyl}.

For $k=1$ and $2$,
let $\L_k=D_k^{-1/2}L_kD_k^{-1/2}$ be the normalized Laplacians corresponding to $L_k$.
We next prove that $\L_1\comp \L_2$.
By Remark~\ref{rem-mat}~(ii),
it suffices to prove that
\begin{equation}\label{cL-D}
|n_+(\L_1-rI)-n_+(\L_2-rI)|\le 1
\end{equation}
for any $r\in\mathbb{R}$.
Recall that eigenvalues of $\L_k$ are in the interval $[0,2]$,
hence $n_+(\L_k-rI)\equiv n$ for $r<0$ and $n_+(\L_k-rI)\equiv 0$ for $r>2$,
and the inequality \eqref{cL-D} is obviously true unless $r\in[0,2]$.
So it remains to consider the case $r\in[0,2]$.
By Lemma~\ref{n=n}~(ii), it suffices to prove that
\begin{equation}\label{L-D}
|n_+(L_1-rD_1)-n_+(L_2-rD_2)|\le 1
\end{equation}
for $r\in[0,2]$.

By the condition \eqref{dif-L},
we may assume that
$$L_k=
\left(\begin{array}{cc}
X_k & Y \\
Y^* & Z \\
\end{array}\right),\qquad k=1,2,$$
where $X_k$ are $2\times 2$ matrices and
$X_1-X_2=w\Delta$.
Let
$$L_k-rD_k=
\left(\begin{array}{cc}
{X}_k(r) & Y \\
Y^* & {Z}(r) \\
\end{array}\right).$$
Then $X_1-X_2=w\Delta$ implies that
${X}_1(r)-{X}_2(r)=w(\Delta-rI)$.
When $\det({Z}(r))\neq 0$,
we have
$$L_k-rD_k
\cong
\left(\begin{array}{cc}
{X}_k(r)-YZ^{-1}(r)Y^* & 0 \\
0 & {Z}(r) \\
\end{array}\right).$$
Clearly, $n_+(L_k-rD_k)=n_+({X}_k(r)-YZ^{-1}(r)Y^*)+n_+({Z}(r))$.
Thus
\begin{equation}\label{n-n}
n_+(L_1-rD_1)-n_+(L_2-rD_2)=n_+({X}_1(r)-YZ^{-1}(r)Y^*)-n_+({X}_2(r)-YZ^{-1}(r)Y^*).
\end{equation}
For two $2\times 2$ matrices ${X}_1(r)-YZ^{-1}(r)Y^*$ and ${X}_2(r)-YZ^{-1}(r)Y^*$,
their difference
$$({X}_1(r)-YZ^{-1}(r)Y^*)-({X}_2(r)-YZ^{-1}(r)Y^*)={X}_1(r)-{X}_2(r)=w(\Delta-rI),$$
which is indefinite for $r\in [0,2]$.
It must be
$$|n_+({X}_1(r)-YZ^{-1}(r)Y^*)-n_+({X}_2(r)-YZ^{-1}(r)Y^*)|\le 1.$$
It follows from \eqref{n-n} that \eqref{L-D} holds when $r\in [0,2]$ and $\det({Z}(r))\neq 0$,
and so that \eqref{cL-D} holds when $\det({Z}(r))\neq 0$.
Clearly, there are only finite real numbers $r\in [0,2]$ such that $\det({Z}(r))=0$.
Hence \eqref{cL-D} holds for all $r\in\mathbb{R}$ by Remark \ref{fin}, as required.
The proof is complete.
\end{proof}

Applying Proposition \ref{dif} to
the Hermitian Laplacian matrix
and Hermitian normalized Laplacian matrix of the second kind of a digraph,
we obtain the following result about Mohar's new kind of Hermitian matrices for digraphs.

\begin{coro}
Let $e$ be an arc or a digon of a digraph $X$.
Then $\w{L}(X-e)\intrz\w{L}(X)$ and $\w{\L}(X-e)\comp\w{\L}(X)$.
\end{coro}

\section{Remarks}

Let $f(x)$ be a real-rooted polynomial.
For $r\in\mathbb{R}$,
let $\n(f,r)$ be the number of roots of $f$ in the interval $(r,+\infty)$
and denote $f_r(x)=f(x+r)$.
Then $\n(f,r)=n_+(f_r)$.
Parallel to Theorem \ref{cri-thm}, we have the following result.

\begin{prop}\label{fgm}
Let $f,g$ be two real-rooted polynomials and $m\in\mathbb{Z}$.
Then $r_{i+m}(g)\le r_{i}(f)$ for all $i$ if and only if
$\n(g,r)-\n(f,r)\le m$ for any $r\in\mathbb{R}$.
\end{prop}

In particular,
$f\intrz g$ if and only if $0\le\n(g,r))-\n(f,r)\le 1$ for any $r\in\mathbb{R}$,
and $f\comp g$ if and only if $|\n(g,r))-\n(f,r)|\le 1$ for any $r\in\mathbb{R}$.
There are two closely related results:
$f\intrz g$ or $g\intrz f$ if and only if $af(x)+bg(x)$ is real-rooted for any $a,b\in\mathbb{R}$,
and $f\comp g$ if and only if $af(x)+bg(x)$ is real-rooted for any $a,b\in\mathbb{R}^+$
(see \cite{CS07} for more information).
Using such a characterization of interlacing polynomials,
Fisk \cite{Fis05} gave a very short proof of the Cauchy interlacing theorem.

Fan and Pall \cite[Theorem 1]{FP57} established the converse of the inclusion principle:
Let $f$ and $g$ be two monic real-rooted polynomials
satisfying $\deg f=\deg g+p$ and $r_{i+p}(f)\le r_i(g)\le r_i(f)$.
Then there is one Hermitian matrix
$A=\left[
   \begin{array}{cc}
     B & C \\
     C^* & D \\
   \end{array}
 \right]$
such that the characteristic polynomials of $A$ and $B$ are $f$ and $g$ respectively.
Wang and Zheng \cite[Theorem 1.3]{WZ19} recently established
a converse of Corollary \ref{weyl}:
Let $f$ and $g$ be two monic real-rooted polynomials with the same degree
satisfying $r_{i+q}(f)\le r_i(g)\le r_{i-p}(f)$.
Then there exist two Hermitian matrices $A$ and $B$
whose characteristic polynomials are $f$ and $g$ respectively,
such that $n_+(B-A)\le p$ and $n_-(B-A)\le q$.
These converse results can be proved by means of Proposition \ref{fgm}.
\section*{Acknowledgement}
This paper has been accepted by the editorial board for publication in Linear and Multilinear Algebra.

The authors thank the anonymous referee for his/her careful reading and valuable suggestions.
This work was supported partially by the National Natural Science Foundation of China (Nos. 11601062, 11771065, 11871304),
the Natural Science Foundation of Shandong Province of China (No. ZR2017MA025),
the Fundamental Research Funds for the Central Universities (No. DUT18RC(4)068),
and the Young Talents Invitation Program of Shandong Province.
\bibliographystyle{amsplain}

\begin{thebibliography}{20}

\bibitem{AM85}N.W. Anderson and T.D. Morley,
Eigenvalues of the Laplacian of a graph,
Linear and Multilinear Algebra 18 (2) (1985), 141--145.

\bibitem{AT14}F.M. Atay and H. Tun\c{c}el,
On the spectrum of the normalized Laplacian for signed graphs:
interlacing, contraction, and replication,
Linear Algebra Appl. 442 (2014) 165--177.

\bibitem{BB08}J. Borcea and P. Br\"{a}nd\'{e}n,
Applications of stable polynomials to mixed determinants: Johnson's conjectures, unimodality, and symmetrized Fischer products,
Duke Math. J. 143 (2008) 205--223.

\bibitem{Butler}S. Butler,
Interlacing for weighted graphs using the normalized Laplacian,
Electron. J. Linear Algebra 16 (2007) 90--98.

\bibitem{Chen04}G. Chen, G. Davis, F. Hall, Z. Li, K. Patel and M. Stewart,
An interlacing result on normalized Laplacians,
SIAM J. Discrete Math. 18 (2) (2004) 353--361.

\bibitem{CS07}M. Chudnovsky and P. Seymour,
The roots of the independence polynomials of a clawfree graph,
J. Combin. Theory Ser. B 97 (2007) 350--357.



\bibitem{FP57}K. Fan and G. Pall,
Imbedding conditions for Hermitian and normal matrices,
Canad. J. Math. 9 (1957) 298--304.


\bibitem{Fis05}S. Fisk,
A very short proof of Cauchy's interlace theorem,
Amer. Math. Monthly, 112 (2005) 118.


\bibitem{GM17}K. Guo and B. Mohar,
Hermitian adjacency matrix of digraphs and mixed graphs,
J. Graph Theory 85 (2017) 217--248.


\bibitem{HJ13}R.A. Horn and C.R. Johnson,
Matrix Analysis, 2nd edition,
Cambridge Univ. Press, Cambridge, 2013.

\bibitem{Li}C.-K. Li,
A short proof of interlacing inequalities on normalized Laplacians,
Linear Algebra Appl. 414 (2006) 425--427.


\bibitem{LL15}J. Liu and X. Li,
Hermitian-adjacency matrices and Hermitian energies of mixed graphs,
Linear Algebra Appl. 466 (2015) 182--207.

\bibitem{Liu12}L.L. Liu,
Polynomials with real zeros and compatible sequences,
Electron. J. Combin. 19 (3) (2012) Paper 33.

\bibitem{LW07}L.L. Liu and Y. Wang,
A unified approach to polynomoal sequences with only real zeros,
Adv. in Appl. Math. 38 (2007) 542--560.

\bibitem{MSS15}A.W. Marcus, D.A. Spielman and N. Srivastava,
Interlacing families I: bipartite Ramanujan graphs of all degrees,
Ann. of Math. (2) 182 (2015) 307--325.

\bibitem{MSS15-2}A.W. Marcus, D.A. Spielman and N. Srivastava,
Interlacing families II: Mixed characteristic polynomials and the Kadison-Singer problem,
Ann. of Math. (2) 182 (2015) 327--350.


\bibitem{MSS18}A.W. Marcus, D.A. Spielman and N. Srivastava,
Interlacing families IV: Bipartite Ramanujan graphs of all sizes,
SIAM J. Comput. 47 (2018) 2488--2509.

\bibitem{Mer95}R. Merris,
A survey of graph Laplacians,
Linear and Multilinear Algebra 39 (1--2) (1995), 19--31.

\bibitem{Moh20}B. Mohar,
A new kind of Hermitian matrices for digraphs,
Linear Algebra Appl.  584 (2020) 343--352.

\bibitem{WY05}Y. Wang and Y.-N. Yeh,
Polynomials with real zeros and P\'olya frequency sequences,
J. Combin. Theory Ser. A 109 (2005) 63--74.

\bibitem{WZ19}Y. Wang and S.-N. Zheng,
The converse of Weyl's eigenvalue inequality,
Adv. in Appl. Math. 109 (2019) 65--73.

\bibitem{YDEJ}G. Yu, M. Dehmer, F. Emmert-Streib and H. Jodlbauer,
Hermitian normalized Laplacian matrix for directed networks,
Inform. Sci. 495 (2019) 175--184.
\end{thebibliography}

\end{document}